\documentclass[12pt, leqno]{article}

\usepackage[margin=1in]{geometry}
\usepackage{amsmath, amssymb, amsthm, mathtools}
\usepackage{mathrsfs}
\usepackage{enumerate}
\usepackage[colorlinks=true,linkcolor=blue,citecolor=blue,urlcolor=blue]{hyperref}
\usepackage{cite}
\usepackage{xcolor}
\usepackage{authblk}

\newtheorem{theorem}{Theorem}[section]

\newtheorem{proposition}[theorem]{Proposition}
\newtheorem{lemma}[theorem]{Lemma}

\theoremstyle{definition}
\newtheorem{definition}[theorem]{Definition}

\newtheorem{question}[theorem]{Question}

\newtheorem{remark}[theorem]{Remark}
\newtheorem{example}[theorem]{Example}

\numberwithin{equation}{section}

\newcommand{\R}{\mathbb R}
\newcommand{\Rn}{\R^n}
\newcommand{\N}{\mathbb N}

\newcommand{\germ}{\mathscr G^+}

\newcommand{\jet}{\mathscr J}
\newcommand{\norm}[1]{\left\lVert #1 \right\rVert}
\newcommand{\abs}[1]{\left\lvert #1 \right\rvert}
\newcommand{\cN}{\mathcal N}

\newcommand{\LA}[1]{\refstepcounter{equation}\text{(\theequation)}\label{#1}}
\newcommand{\LAQ}[2]{\begin{itemize}\item[\LA{#1}]{#2} \end{itemize}}

\DeclareMathOperator{\dist}{dist}

\DeclareMathOperator{\ord}{ord}

\newcommand{\set}[1]{\left\{#1\right\}}
\newcommand{\void}{\varnothing}
\newcommand{\brac}[1]{\left(#1\right)}
\newcommand{\eqindent}{\displayindent0pt\displaywidth\textwidth}

\renewcommand{\epsilon}{\varepsilon}
\newcommand{\eps}{\epsilon}
\renewcommand{\phi}{\varphi}

\newcommand{\ring}{\mathcal R}
\newcommand{\cone}{\mathfrak C}
\newcommand{\sa}{\mathcal S}
\newcommand{\sat}{\operatorname{Sat}}
\newcommand{\module}{\mathfrak M}
\newcommand{\qmodule}{\mathfrak Q}
\newcommand{\lin}{\mathfrak L}
\renewcommand{\subset}{\subseteq}

\title{Semialgebraic ${C^m}$ solutions to a univariate system of linear inequalities with Nash forcing}

\author[1]{Fushuai Jiang}
\author[2]{Garving K. Luli}
\affil[1]{Department of Mathematics, City University of Hong Kong}
\affil[2]{Department of Mathematics, University of California - Davis}
\date{}

\begin{document}

\maketitle

\begin{abstract}
    Let $(A_{ij}(x))_{1 \leq i \leq N, 1 \leq j \leq M}$ be a matrix of univariate semialgebraic functions. We study the solvability of the system of linear inequalities
    \begin{equation*}
        \sum_{j=1}^M A_{ij}(x)F_j(x)\leq f_i(x)
        \text{ \quad for } i = 1, \ldots, N,
    \end{equation*}
    where we assume each forcing $f_i$ is a Nash function. We show that the $C^m$ solvability of the system can be characterized by finitely many linear ordinary differential inequalities in $f$ with semialgebraic coefficients depending only on $(A_{ij})$.  When the system is solvable, the solution may be chosen semialgebraic, and the proof gives a constructive procedure in the semialgebraic sense. We also study the semimodule structure of admissible forcings, establishing finite-generation and lack thereof in various settings. 
\end{abstract}

\section{Introduction}

Fix integers $m \geq 0$, $M,N \geq 1$. For $1 \leq i \leq N$, $1 \leq j \leq M$, let $A_{ij}:\R \to \R$ be fixed semialgebraic functions, i.e., their graphs are finite unions of sets defined by polynomial equalities and inequalities~\cite{BasuPollackRoy2006}. Given $f = (f_1,\cdots, f_N) \in C^\infty(\R,\R^N)$, we ask whether there exists $F = (F_1,\cdots, F_M)\in C^m(\R,\R^M)$ that solves the system
\begin{equation}
    \sum_{j=1}^M A_{ij}(x)F_j(x) \leq f_i(x)
    \text{\quad for all $x \in \R$ and $i = 1, \cdots, N$.}
    \label{eq:system main}
\end{equation}
We will adopt the terminology from~\cite{brenner2006continuous} and call $f = (f_1, \cdots, f_N)$ a \emph{forcing vector}, or simply forcing. 

In this paper, we show in Theorem~\ref{thm:main} that, assuming each $f_i$ is Nash, the $C^m$ solvability of~\eqref{eq:system main} can be characterized by finitely many local linear-differential sign conditions on the data. When the system~\eqref{eq:system main} is solvable, the solution can be taken to be semialgebraic, and the proof gives a constructive procedure.

\medskip

To lay some background on the problem at hand, let us review an example of Epstein and Hochster~\cite{EpsteinHochster2018}. Consider the single  \emph{linear equation}
\begin{equation}\label{eq:Epstein Hochster example}
    x^2F_1+y^2F_2+xyz^2F_3=f(x,y,z).
\end{equation}
There exist continuous $F_1,F_2,F_3$ (no differentiability required) satisfying \eqref{eq:Epstein Hochster example} if and only if
\begin{equation}
    \begin{cases}
        f\left( x,y,z\right) =\frac{\partial f}{\partial x}\left( x,y,z\right) = \frac{\partial f}{\partial y}\left( x,y,z\right) =0 
        &\text{ for $x = y = 0$ and $z \in \R$}
        \\
        \frac{\partial ^{2}f}{\partial x\partial y}\left( x,y,z\right) =\frac{\partial ^{3}f}{\partial x\partial y\partial z}\left( x,y,z\right) =0 
        &\text{ at $x = y = z = 0$}
    \end{cases} \ .
    \label{eq:EpsteinHochster}
\end{equation}

The differential \emph{equality} characterization in~\eqref{eq:EpsteinHochster} is no coincidence. In~\cite{FeffermanLuli2021Equations,FeffermanLuli2021Generators}, C.~Fefferman and the second author proved that the corresponding problem for semialgebraic systems of \emph{linear equations} on $\Rn$ 
\begin{equation}
    \sum_{j=1}^M A_{ij}(x)F_j(x) = f_i(x)
    \text{\quad for all $x \in \Rn$ and $i = 1,\cdots, N$,}
    \label{eq:system equations}
\end{equation}
has a finite differential criterion. 

\begin{theorem}[Fefferman-Luli~\cite{FeffermanLuli2021Equations,FeffermanLuli2021Generators}]
\label{thm:feffermanluli}
Fix $m\geq 0$, and let $(A_{ij}(x))_{1\leq i\leq N,\ 1\leq j\leq M}$ be a matrix of semialgebraic functions on $\mathbb R^n$. Then there exist $\bar m\geq m$ and a finite list of linear partial differential operators
\[
L_\nu f(x)
=
\sum_{i=1}^N
\sum_{|\alpha|\leq \bar m}
a_{\nu i\alpha}(x)\partial^\alpha f_i(x), \quad \nu = 1, \cdots, \nu_{\max},
\]
with semialgebraic coefficient functions $a_{\nu i\alpha}$, such that~\eqref{eq:system equations} admits a solution $F\in C^m(\mathbb R^n,\mathbb R^M)$ if and only if $L_\nu f \equiv 0$ for every $\nu = 1, \cdots, \nu_{\max}$.
\end{theorem}

We give a brief account of recent progress leading to Theorem~\ref{thm:feffermanluli}, which is closely related to the classical theory of ideals of differentiable functions~\cite{whitney1948ideals}.
For systems of \emph{linear equations}, continuous ($C^0$) solvability leads to the continuous closure of ideals. Brenner introduced algebraic tests through axes closure~\cite{brenner2006continuous}; 
J.~Koll\'ar developed a purely algebraic construction of continuous closure for torsion-free modules~\cite{kollar2012continuous-closure}; 
Fefferman and Koll\'ar provided an alternative solution combining Whitney extension techniques with an algebraic descent procedure~~\cite{FeffermanKollar2013Continuous}. 
As a companion result in the $C^m$ setting, Fefferman and Luli further gave an algorithm for computing its generators \cite{FeffermanLuli2021Generators}. See also~\cite{bierstone2026curve} for a concurrent one-dimensional result.

See \cite{BierstoneCampesatoMilman2021,Fefferman2006WhitneyCm,FeffermanIsraelLuli2016Selection,FeffermanLuli2022Plane,JiangLuliONeill2022Shape,LuliONeill2023Inequalities,JiangLuliONeill2022Smooth, brenner2006continuous,kollar2012continuous-closure,FeffermanKollar2013Continuous,kucharzkurdyka2017linear,kollarnowak2015rational,bierstone2026curve,FeffermanLuli2014BHK} and the references therein for more details.

\medskip

In \cite{LuliONeill2023Inequalities}, the second author and K.~O'Neill ask for an inequality analogue of this theorem. They show that such a criterion fails in general for continuous solutions in dimensions $n\geq 2$. The root cause is that Glaeser refinement in dimension $n\geq 2$ may not preserve polygonality.

The corresponding problem for one dimension was left open and is fundamentally different: there are only two directions to take limits, and univariate semialgebraic functions have finite Puiseux asymptotics. 

In our concurrent paper~\cite{jiangluli2026semialgebraic}, we show that in the univariate setting, such a characterization is possible if we seek a $C^m$ solution for $m \in \set{0,1,2}$ without the Nash assumption on the $f_i$'s. In that paper, we also showed as a pleasant surprise that the analogous result fails for $m > 2$, as jets can no longer capture small-amplitude oscillations. The dichotomy between $C^2$ and $C^{m>2}$ is also a recurring theme in the authors' previous works on nonnegative extension~\cite{jiang2024flat,jiangluli2020-finiteness,jiangluli2020-operator,jiangluli2021-alg}.

However, with an additional rigidity assumption (such as Nash in this paper) on the forcing vector $f$, such a characterization is still possible for general $C^m$.

\subsection{Main results}

For the rest of the paper, we use $\cN(I,\R^N)$ to denote the space of maps on an open semialgebraic interval $I$ whose components are Nash, i.e., $C^\infty$ and semialgebraic. 

The main theorem of the paper is as follows. 

\begin{theorem}\label{thm:main}
    Let $(A_{ij}(x))_{1\leq i \leq N, 1 \leq j \leq M}$ be a matrix of univariate semialgebraic functions. For each $m \in \mathbb N_0$, there exists an integer $\bar m \geq m$ and a finite semialgebraic partition $\mathcal C$ of $\R$ such that on each cell $C \in \mathcal C$, there are 
    \begin{itemize}
        \item finitely many linear ordinary differential expressions
    \begin{equation}
    \eqindent
        L_{C,\nu}f(x) = \sum_{i=1}^N \sum_{r=0}^{\bar m} a_{C,\nu ir}(x)f_i^{(r)}(x)
        \text{\quad,\quad}
        \nu = 1, \cdots, \nu_{\max,C}
    \end{equation}
    whose coefficients are semialgebraic on $C$, and
    \item a finite Boolean formula $\mathcal S_C$ in the sign conditions 
    \begin{equation*}
    \eqindent
        \set{L_{C,\nu}f \ \Omega \ 0: \Omega \in \set{>,=,<}, \ C \in \mathcal C,\ \nu = 1, \cdots, \nu_{\max,C}}
    \end{equation*}
    \end{itemize}
    such that for every Nash vector $f \in \cN(\R,\R^N)$, the following are equivalent.
    \begin{enumerate}[(A)]
        \item There exists a $F \in C^m(\R,\R^M)$ satisfying $A(x)F(x)\leq f(x)$ for all $x \in \R$.
        \item For every $C \in \mathcal C$ and every $x \in C$, the formula $\mathcal S_C(L_{C,1}f(x), \ldots, L_{C,\nu_{\max,C}}f(x))$ holds. 
    \end{enumerate}
    Moreover, if (B) holds, the function $F$ in (A) can be taken to be semialgebraic.
    The partition, the coefficient functions, Boolean formulas, and the integer $\bar m$ depend only on $(A_{ij})$, $M$, $N$, and $m$.
\end{theorem}

The proof also gives an effective semialgebraic construction of the solution if it exists. In particular, see Proposition~\ref{prop:nash weak lattice} and Remark~\ref{rem:curve-selection-constructive}.

\medskip

The differential description in Theorem~\ref{thm:main} also leads to an algebraic description of the admissible forcings on each regular interval $I$ in the sense of Definition~\ref{def:regular interval exception point} below. For the moment, it is harmless to pretend $I\subset \R $ is just an open and semialgebraic interval.

To start with, we define three objects.
\begin{itemize}
    \item Let the cone of admissible forcing vectors be
\begin{equation*} 
\eqindent
\cone_{A,I}^m(\cN(I)) := \set{ f\in\cN(I,\R^N): 
\text{ there exists }F\in C^m(I,\R^M) \text{ such that }A F\leq f }. 
\end{equation*} 
\item Let the lineality module of the admissible cone be
\begin{equation*}
\eqindent
    \lin_{A,I}^m(\cN(I)) 
:= \cone_{A,I}^m(\cN(I)) \cap \brac{-\cone_{A,I}^m(\cN(I))}.
\end{equation*}
\item Let $\cN(I)$ be the underlying ring of Nash functions on $I$ and let
\begin{equation*}
\eqindent
    \cN_+(I) := \set{ q\in\cN(I): q\geq0\text{ on }I }
\end{equation*}
 be the underlying semiring.
\end{itemize}

We will prove the following.

\begin{theorem}\label{thm:finite generation regular interval} 

Let $I\subset\R\setminus\Sigma_A$ be a regular interval in the sense of Definition~\ref{def:regular interval exception point}. 
Then $\lin_{A,I}^m(\cN(I))$ is a finitely generated $\cN(I)$-module. Moreover, the image of $\cone_{A,I}^m(\cN(I))$ in $\cN(I)^N \Big/ \lin_{A,I}^m(\cN(I))$ is generated as a semimodule over $\cN_+(I)$ by the coordinate classes $[e_1],\ldots,[e_N]$. 
Consequently, there exist $u_1,\ldots,u_r\in\cN(I)^N$ 
such that \begin{equation*} 
\cone_{A,I}^m(\cN(I)) = \sum_{\alpha=1}^r\cN(I)u_\alpha + \sum_{i=1}^N\cN_+(I)e_i. \end{equation*} 
\end{theorem} 

The proof is based on a slack-variable representation $AF + \sigma = f$ of the system~\eqref{eq:system main}. 

The Nash hypothesis in Theorem~\ref{thm:finite generation regular interval} is essential. As we will see in Theorem~\ref{thm:quotient module polynomial not FG} below, finite generation may fail if the base ring is replaced by the restricted polynomial ring $\R[x]|_I$ while the matrix $A$ is allowed to have nonrational Nash entries.

\subsection{Structure of the paper}

Now we lay out the proof strategy of Theorem~\ref{thm:main} and the organization of this paper. 

In Section~\ref{sect:regular intervals and lattice}, we apply formal Fourier-Motzkin elimination to the fixed matrix $(A_{ij})$ and obtain a finite semialgebraic decomposition of $\R$ into regular intervals and exceptional points (Definition~\ref{def:regular interval exception point}). On regular intervals, the coefficients that arise in the elimination are smooth, and the pointwise nonemptiness of the solution fiber $K_f(x) = \set{v \in \R^M : A(x)v \leq f(x)}$ implies the existence of a $C^m$ selection. This step makes use of the assumption that $f$ is of Nash class, which admits a weak lattice structure (Section~\ref{sect:weighted weak nash lattice}). 

The entire Section~\ref{sect:exceptional points} is dedicated to the main compatibility issue at an exceptional point, where the coefficient from the Fourier-Motzkin elimination becomes singular. We give the main ingredients in Section~\ref{sect:weighted weak nash lattice}, which include a weighted version of the Nash weak lattice and a Nash selection with controlled order drop.
We give the punchline in Section~\ref{sect:extension at exceptional point}, which roughly states that an $m$-jet is feasible at an exceptional point $y\in \R$ if and only if $P(y) \in K_f(y)$ and the distance from $P(y+t)$ to $K_f(y+t)$ must be controlled asymptotically by $o(t^m)$ as $t \downarrow 0$.

In Section~\ref{sect:proof of main theorem}, we prove Theorem~\ref{thm:main} by gluing together solutions on regular intervals with solutions at exceptional points. The finite-Boolean criteria and semialgebraicity are preserved at every stage of the construction.

The proof of Theorem~\ref{thm:main} is constructive in the semialgebraic sense. On regular intervals, the Fourier-Motzkin projection and the Nash selection (Proposition~\ref{prop:nash weak lattice}) can be given explicitly. Near an exceptional point, the semialgebraic curve selection can be made effective. See Remark~\ref{rem:curve-selection-constructive}.

Finally, we study the algebraic structure of the admissible forcing vectors in Section~\ref{sect:generators}.

On each regular interval, the admissible Nash data form a semimodule over the semiring of nonnegative Nash functions. We prove Theorem~\ref{thm:finite generation regular interval}, i.e., the lineality space is a finitely generated Nash module, while the quotient by lineality is generated by the coordinate classes. 

We then examine the corresponding problem over the restricted polynomial ring. Theorem~\ref{thm:quotient module polynomial not FG} shows that finite generation can fail when the matrix has a nonrational Nash entry. This leaves open the more rigid algebraic case in which the matrix itself has polynomial entries, the \emph{system-of-equations} analogue of which has been an object of intensive study. See, e.g.,~\cite{kollarnowak2015rational,EpsteinHochster2018,kucharzkurdyka2017linear}. 

The section concludes with an additional Example~\ref{example:qjet}, where we provide a quadratic ring in two variables and show that even an admissible cone need not be finitely generated. This is regardless of whether the cone has a Boolean-finite description (which is already not always possible by~\cite{LuliONeill2023Inequalities}). This suggests that the univariate Nash result is close to the natural limit of the present method, while the case of a polynomial matrix over the restricted polynomial ring remains a natural algebraic problem. We collect this algebraic problem along with the issue of effective computation of generators in Question~\ref{quest:effective algebraic generators} at the end of the paper.

\section{Weak lattice and regular intervals}
\label{sect:regular intervals and lattice}

In this section, we define regular intervals and exceptional points (Definition~\ref{def:regular interval exception point}). Exceptional points occur when the coefficients of the Fourier-Motzkin elimination become singular and require separate treatment in the next section. On regular intervals, data are smooth (and hence Nash), and we show that a $C^m$ selection can be constructed using a weak lattice argument in the spirit of Nagel and Rudin~\cite{NagelRudin1974}.

\subsection{Univariate semialgebraic limits}
\label{sect:semialgebraic limit}

Fix $x_0 \in \R$. We use $\germ_{x_0}$ to denote the ring of right-hand semialgebraic germs at $x_0$. For the standard theory of one-sided semialgebraic germs and their identification with algebraic Puiseux series, we refer to~\cite[Section~3.3]{BasuPollackRoy2006}.

We recall an elementary result from~\cite[Lemma 3.1]{jiangluli2026semialgebraic}. 
Let $\phi\in\mathscr G^+_{x_0}$ be nonzero. Then there exist $c\neq0$ and $\gamma\in\mathbb Q$ such that
        \begin{equation*}
        \eqindent
            \phi(x) = c(x-x_0)^\gamma + o\brac{(x-x_0)^\gamma}
            \text{\quad as $x \downarrow x_0$.}
        \end{equation*}
In particular, the sign of $\phi(x)$ is eventually constant as $x\downarrow x_0$. As such, we may define the order of a semialgebraic right-hand germ to be
\begin{equation}
    \eqindent
    \ord_{x_0}(\phi)=\begin{cases}
        \gamma &\text{ if $\phi \in \germ_{x_0}$ is nonzero}
        \\
        +\infty &\text{ otherwise}
    \end{cases} \ .
    \label{eq:order def}
\end{equation}

\subsection{Weak lattice of smooth functions}
\label{sect:weak lattice}

Though the central assumption of this paper is the forcing $f$ being Nash, we introduce a slightly more general definition that potentially incorporates a wide class of functions. See Remark~\ref{rem:general weak lattice selection} after the main Proposition~\ref{prop:solution on regular interval}.

\newcommand{\wl}{\mathsf{WL}}
\begin{definition}\label{def:weak lattice}
    Let $I\subset \R$ be an open interval. 
    We say a vector space of continuous functions $\mathbb X(I)$ is a $\wl$ space if it satisfies the weak-lattice and closure properties. 
    \LAQ{eq:weak lattice}{
    If $L_1,L_2, U_1,U_2 \in \mathbb X(I)$ and $L_i \leq U_j$ for $i,j = 1,2$, then there exists $h \in \mathbb X(I)$ such that $L_i \leq h \leq U_j$ for $i,j = 1,2$.
    }
    \LAQ{eq:closure}{$\mathbb X(I)$ is closed under the operations $\phi\mapsto \phi^2$ and $\phi\mapsto\sqrt{1+\phi^2}$.}
\end{definition}

\begin{lemma}\label{lem:weak lattice for arbitrary family}
    Let $\mathbb X(I)$ be a $\wl$ space. 
    Let $p,q \in \mathbb N_{\geq 0}$. 
    If $L_1, \cdots, L_p, U_1, \cdots, U_q \in \mathbb X(I)$ satisfy 
    \begin{equation*}
        \max_{1\leq i \leq p}L_i \leq \min_{1\leq j \leq q}U_j, 
    \end{equation*}
    then there exists $h \in \mathbb X(I)$ such that 
    \begin{equation*}
        \max_{1\leq i \leq p}L_i \leq h \leq \min_{1\leq j \leq q}U_j.
    \end{equation*} 
    If $p$ or $q$ is zero, the respective lower or upper bound is removed. 
\end{lemma}

\begin{proof}
    First, we treat the case $p,q > 0$.
    If $p = q = 1$, the conclusion is trivial. Assume, for the moment, that $p = 2$. We will induct on $q$. 
    
    If $q = 1$, we simply take $h = U_1$. Assume inductively that there exists $h_{r-1}\in \mathbb X(I)$ such that
    \begin{equation}
        \max_{i=1,2}L_i \leq h_{r-1}
        \leq
        \min_{1\leq j \leq r-1}U_j.
        \label{eq:hr-1}
    \end{equation}
    By the definition of $\mathbb X(I)$ applied to the quadruple $(L_1,L_2,h_{r-1}, U_r)$, there exists $h_r \in \mathbb X(I)$ such that
    \begin{equation*}
        L_1,L_2 \leq h_r \leq h_{r-1}, U_r.
    \end{equation*}
    Since $h_r \leq h_{r-1}$, we see from~\eqref{eq:hr-1} that
    \begin{equation*}
        L_1, L_2 \leq h_r \leq U_1, \cdots, U_r.
    \end{equation*}
    The induction on the number $p$ of lower bounds is similar, so we omit the details.

    If $q = 0$, we may define $h = \sum_i \sqrt{1+L_i^2}$. Since each $L_i^2, \sqrt{1+L_i^2}\in \mathbb X(I)$ by assumption~\eqref{eq:closure} and $\mathbb X(I)$ is an additive group, we see that $h \in \mathbb X(I)$ and $h\geq L_i$ for every $i$. If $p = 0$, we define $h = -\sum_j \sqrt{1+U_j^2} \in \mathbb X(I)$, which satisfies $h \leq U_j$ for every $j$. 

    This completes the proof.
\end{proof}

In~\cite{NagelRudin1974}, Nagel and Rudin showed that $C^m$ is a $\wl$ space if and only if $m \in \set{0,1,2}$. For higher-order smoothness, more rigidity is required to eliminate small-amplitude ososcillations

\begin{proposition}\label{prop:nash weak lattice}
    Let $I\subset \R$ be an open and semialgebraic interval. Then $\cN(I)$ is a $\wl$ space. 
\end{proposition}

\begin{proof}
Suppose $L_1,L_2, U_1, U_2 \in \cN(I)$ with $L_i \leq U_j$ for $1 \leq i,j \leq 2$. Set $f_1 := U_2 - L_1$, $f_2 = U_1 - L_2$, and $g:= U_1 - L_1$. Then $f_1, f_2, g \geq 0$ and
\begin{equation}
    0 \leq g \leq f_1 + f_2.
    \label{eq: g bounded by f1 + f2}
\end{equation}
We proceed as the Nagel-Rudin interpolation~\cite{NagelRudin1974} and define
\begin{equation*}
    g_i(x) := \begin{cases}
        \frac{f_i(x)g(x)}{f_1(x) + f_2(x)}
        &\text{ if $f_1(x) + f_2(x) > 0$}
        \\
        0 &\text{ otherwise}
    \end{cases}\, 
    \text{\quad for $i = 1,2$.}
\end{equation*}
Thanks to \eqref{eq: g bounded by f1 + f2}, we have
\begin{equation*}
    g_1 + g_2 = g
    \text{\quad and \quad}
    0 \leq g_i \leq f_i \text{\quad for $i = 1,2$.}
\end{equation*}
We claim that $g_1,g_2 \in \cN(I)$. It suffices to show that the formal singularities at the zeroes of $f_1 + f_2$ are removable.

Fix $x_0 \in I$ such that $f_1(x_0) + f_2(x_0) = 0$. If $f_1 + f_2$ vanishes identically on a neighborhood of $x_0$, then $f_1 \equiv f_2 \equiv g \equiv 0 $ in that neighborhood, and there is nothing to prove. For the rest of the proof, assume otherwise. Since $f_1 + f_2 \in \cN(I)$, hence real analytic on $I$, there exists an integer $r \geq 1$ and a real analytic function $a$ with $a(x_0)\neq 0$ such that
\begin{equation*}
    f_1(x) + f_2(x) = (x-x_0)^ra(x)
    \text{ \quad in a neighborhood of $x_0$.}
\end{equation*}
Since $f_1 + f_2 \geq 0$, the exponent $r$ must be even and $a(x_0) > 0$. The condition $0 \leq f_1, g \leq f_1 + f_2$ forces $\ord_{x_0}(f_1), \ord_{x_0}(g) \geq r$, and therefore,
\begin{equation*}
    \ord_{x_0}(f_1g) \geq 2r
    .
\end{equation*}
This means that $\frac{f_1g}{f_1+f_2}$ can be extended real analytically across $x_0$ and its extension vanishes at $x_0$. Moreover, this extension must agree with $g_1$. Therefore, $g_1$ is real analytic on $I$. The same argument applies to $g_2$.

We claim that $g_1,g_2 \in \cN(I)$. To see this, we note that
\begin{itemize}
    \item on the semialgebraic set $\set{f_1 + f_2} > 0$, they are quotients of Nash functions;
    \item on the semialgebraic set $\set{f_1 + f_2 = 0}$, they are identically zero.
\end{itemize}

For the last step, define $h:= L_1 + g_1 \in \cN(I)$. It is clear from construction that $L_1,L_2 \leq h \leq U_1,U_2$. This completes the proof.
\end{proof}

\subsection{Regular intervals}

We begin with an elementary result from~\cite[Corollary 2.2]{jiangluli2026semialgebraic}

\begin{proposition}\label{prop:Fourier Motzkin}
    After a finite semialgebraic partition of $\R$, the condition 
        \begin{equation*}
            K_f(x) = \set{v \in \R^M : A(x)v\leq f(x)} \neq \void
        \end{equation*}
        is equivalent on each cell to finitely many homogeneous linear inequalities in $f(x)$ with semialgebraic coefficients determined by $A$.
\end{proposition}

Proposition~\ref{prop:Fourier Motzkin} reduces the nonemptiness of $K_f(x)$ on each cell of a finite semialgebraic partition to finitely many homogeneous linear inequalities whose coefficients are semialgebraic functions constructed from $A$. However, to apply the Lemma~\ref{lem:weak lattice for arbitrary family} and Proposition~\ref{prop:nash weak lattice}, we need these coefficients to be Nash and of constant sign, so that no coefficient changes role or degenerates. Since only finitely many such coefficients occur, and each is semialgebraic, the set where this regularity fails is finite. Refining the partition accordingly leads to the following definition.

\begin{definition}\label{def:regular interval exception point}
    Let $\Sigma_A \subset \R$ be a finite set such that on every connected component of $\R\setminus \Sigma_A$, every coefficient arising in the formal Fourier-Motzkin elimination is smooth (hence Nash) and has a constant sign in $\set{+, -, 0}$. We call each connected component of $\R\setminus \Sigma_A$ a \textbf{regular intervals}; we call each point in $\Sigma_A$ an \textbf{exceptional point}.
\end{definition}

We now construct a Nash, and hence $C^m$, selection on each regular interval. The corresponding extension to suitable smooth $\wl$ spaces and definable classes is discussed in Remark~\ref{rem:general weak lattice selection}.

\begin{proposition}\label{prop:solution on regular interval}
    Let $I\subset \R$ be a regular interval in the sense of Definition \ref{def:regular interval exception point}. Let $f \in \cN(I,\R^M)$ and suppose 
    \begin{equation*}
        K_f(x) = \set{v \in \R^M : A(x)v\leq f(x)} \neq \void
        \text{\quad for all $x \in I$.}
    \end{equation*}
    Then there exists $F \in \cN(I,\R^M)$ such that $A(x)F(x)\leq f(x)$ for all $x \in I$. In particular, $K_f$ admits a $C^m$ selection over $I$. 
\end{proposition}

\begin{proof}
    We induct on the dimension $M$. The base case $M = 0$ is trivial. Suppose the conclusion holds for $1, \cdots, M-1$. Write $v = (v',s)$ with $v' \in \R^{M-1}$ and write the $i$-th inequality in $K_f(x)$ as
    \begin{equation*}
        a_i(x)s + c_i(x) \cdot v' \leq f_i(x).
    \end{equation*}
    Since $I$ is a regular interval, each coefficient $a_i$ has a constant sign (or identically zero) on $I$. Accordingly, define
    \begin{equation*}
        \mathcal I_- := \set{i : a_i < 0}
        \ , \
        \mathcal I_+ := \set{i : a_i > 0}
        \ , \text{ and } \ 
        \mathcal I_0 := \set{i : a_i \equiv 0}.
    \end{equation*}
    \begin{itemize}
        \item For $i \in \mathcal I_-$, we have the lower bound
        \begin{equation*}
            \eqindent
            s \geq \frac{f_i(x)-c_i(x)\cdot v'}{a_i(x)} =: L_i^\sharp(x,v').
        \end{equation*}
        \item For $j \in \mathcal I_+$, we have the upper bound
        \begin{equation*}
            \eqindent
            s \leq \frac{f_j(x)-c_j(x)\cdot v'}{a_j(x)} =: U_j^\sharp(x,v').
        \end{equation*}
        
        \item For $k \in \mathcal I_0$, the inequalities are independent of $s$ and is given by
        \begin{equation*}
            \eqindent
            c_k(x) \cdot v' \leq f_k(x).
        \end{equation*}
    \end{itemize}
    It follows that the Fourier-Motzkin projection of the original system $A(x)F(x)\leq f(x)$ onto the first $M-1$ variables is
    \begin{equation}
        \begin{cases}
            L_i^\sharp(x,v') \leq U_j^\sharp(x,v') &\text{ for all $i \in \mathcal I_-$ and $j \in \mathcal I_+$}
            \\
            c_k(x) \cdot v' \leq f_k(x) &\text{ for all $k \in \mathcal I_0$}
        \end{cases}\, . 
        \label{eq:projected system}
    \end{equation}
    Since the original system is pointwise feasible on $I$, the projected system~\eqref{eq:projected system} is also pointwise feasible on $I$. Moreover, the system~\eqref{eq:projected system} has Nash data by construction. 

    By the induction hypothesis, the system~\eqref{eq:projected system} admits a selection $\bar F \in \cN(I,\R^{M-1})$. Define the scalar functions
    \begin{equation}
        L_i(x):= L_i^\sharp(x,\bar F(x))
        \text{\quad and \quad}
        U_j(x):= U_j^\sharp(x,\bar F(x)).
        \label{eq:Li and Uj}
    \end{equation}
    Since $f\in \cN(I,\R^M)$, $\bar F \in \cN(I,\R^{M-1})$, and the fixed semialgebraic coefficients $a_i, c_i$ are Nash on the regular interval with constant signs, we see that
    \begin{equation}
        L_i, U_j \in \mathbb \cN(I)
        \text{\quad for $i \in \mathcal I_-$ and $j \in \mathcal I_+$.}
        \label{eq: Li Uj Cm}
    \end{equation}
    Since $\bar F$ satisfies the exact Fourier-Motzkin projection, \eqref{eq:projected system} and \eqref{eq:Li and Uj} implies
    \begin{equation}
        \max_{i \in \mathcal I_-}L_i \leq \min_{j\in \mathcal I_+}U_j.
        \label{eq:LiUj cross bound}
    \end{equation}
    In view of \eqref{eq: Li Uj Cm}--\eqref{eq:LiUj cross bound}, we can apply Lemma~\ref{lem:weak lattice for arbitrary family} to $L_i,U_j$ to find a function $F_M \in \cN(I)$ such that
    \begin{equation*}
        \max_{i \in \mathcal I_-}L_i \leq F_M \leq \min_{j\in \mathcal I_+}U_j,
    \end{equation*}
    with the understanding that if either $\mathcal I_-$ or $\mathcal I_+$ is empty, the constraint is removed. 
    Set $F:= (\bar F,F_M)$. It is clear that $F \in \cN(I,\R^M)$ Each $k$-th row for $k \in \mathcal I_0$ is satisfied, since $\bar F$ satisfies the projected system. Rows corresponding to $\mathcal I_\pm$ are satisfied by construction. Therefore, $A(x)F(x)\leq f(x)$ on $I$. This completes the proof.
\end{proof}

\begin{remark}\label{rem:general weak lattice selection} 

The Nash class in Proposition~\ref{prop:solution on regular interval} may be replaced by a more general class of smooth functions. More precisely, let $\mathbb X(I)\subset C^\infty(I)$ be a $\wl$ space in the sense of Definition~\ref{def:weak lattice} that is also a unital algebra and is closed under reciprocals of nowhere-vanishing elements, i.e.,  
\begin{equation*}
    g\in\mathbb X(I) \quad\text{and}\quad g(x)\neq0\ \text{for every }x\in I \quad\Longrightarrow\quad \frac1g\in\mathbb X(I).
\end{equation*}
Suppose, in addition, that all coefficient functions and the forcings arising in the formal Fourier-Motzkin elimination belong to $\mathbb X(I)$. Then the proof of Proposition~\ref{prop:solution on regular interval} applies verbatim. In particular, the algebra and reciprocal-closure assumptions ensure that $\frac{f_i-c_i\cdot\bar F}{a_i} \in \mathbb X(I)$ whenever the pivot coefficient $a_i$ is nowhere vanishing. 
The $\wl$ property, together with Lemma~\ref{lem:weak lattice for arbitrary family}, then provides the final component of the selection between the resulting finite families of lower and upper bounds. 

The same argument applies to any fixed definable class and definable matrix $A$, after partitioning into regular intervals on which the Fourier-Motzkin coefficients retain the required smoothness, sign, and algebraic closure properties.
\end{remark}

\section{Compatibility at exceptional points}
\label{sect:exceptional points}

\subsection{Distance to the feasible polygon}

Throughout the rest of this section, we fix an exceptional point $y \in \Sigma_A$ (recall Definition~\ref{def:regular interval exception point}). We characterize $m$-jets which can be realized by feasible one-sided $C^m$ selections.

As before, for each $x \in \R$,
\begin{equation*}
    K_f(x) = \set{v \in \R^M : A(x)v \leq f(x)}.
\end{equation*}
Suppose a regular interval lies immediately to either the right or the left of $y$. For $P\in\mathcal P^m(\mathbb R,\mathbb R^M)$, define
\begin{equation}
d^\pm_{y,P}(t)
=
\operatorname{dist}_\infty
\brac{P(y\pm t),K_f(y\pm t)}
\text{\quad for $t > 0$,}
\label{eq:distance-plus}
\end{equation}
where $\operatorname{dist}_\infty(v,K) := \inf\limits_{w \in K}\norm{v-w}_{\infty}$ is the distance function with respect to the $\ell_\infty$ norm on a vector space.

We collect two useful results from~\cite[Corollaries 4.2 and 4.3]{jiangluli2026semialgebraic}. 

\begin{proposition}\label{prop:dist to polygon}
The following are true.
\begin{enumerate}[(A)]

\item\label{prop:dist to polygon point} Let $A(t)$ be a matrix of elements in $\germ_0$. Then there exist constants $C>0$ and $\mathbb Q \ni D_{\mathrm{dist}}\geq0$, depending only on $A$, such that whenever $b\in \mathbb R^N$ and
\begin{equation*}
    K_b(t)
:=
\{u\in\mathbb R^M:A(t)u\leq b\} \neq \void
\end{equation*}
we have
\begin{equation*}
\eqindent
    \operatorname{dist}_\infty\brac{v,K_b(t)}\leq 
Ct^{-D_{\mathrm{dist}}} \norm{(A(t)v-b)_+}_{\infty}.
\end{equation*}
for all sufficiently small $t>0$ and all $v\in\R^M$.

\item\label{prop:dist to polygon boolean} There is an integer $s_y$ such that $d^\pm_{y,P}(t) = o(t^m)$ as $t \downarrow 0$ is equivalent to a finite Boolean combination of homogeneous linear sign conditions in $(P,\jet_y^{(s_y)}f)$. Moreover, for some rational $\eta > 0$ depending only on $m$ and the fixed coefficient germs, 
\begin{equation*}
    \eqindent
    d^\pm_{y,P}(t) = o(t^m)
    \quad\implies\quad
    d^\pm_{y,P}(t) = O(t^{m+\eta})
    \text{\quad as $t \downarrow 0$.}
\end{equation*}
\end{enumerate}
\end{proposition}

\subsection{Weighted weak lattice of Nash functions}
\label{sect:weighted weak nash lattice}

Recall that $\germ_0$ denotes the ring of right-hand semialgebraic germs at $0$. After restricting to a sufficiently small punctured right-neighborhood $(0,\eps)$, every element of $\germ_0$ admits a Nash representative.

Let $\lambda \in \R$. We use $\Phi^{m}_\lambda$ to denote the class of germs represented by functions $u \in C^m((0,\eps))$ satisfying
\begin{equation}
    \abs{u^{(r)}(t)} \leq Ct^{\lambda-r}
    \text{\quad , \quad}
    0 \leq r \leq m,
\end{equation}
for some constant $C$ and for all sufficiently small $t > 0$. If $\lambda > m$, then every $u \in \Phi^{m}_\lambda$ extends to a $C^m$ function at $0$ with a zero $m$-jet. 

If $a \in \germ_0\setminus\set{0}$, we have
\begin{equation}
    a\cdot \Phi^{m}_\lambda \subset \Phi^{m}_{\lambda + \ord(a)}
    \text{\quad and \quad}
    a^{-1}\cdot \Phi^{m}_\lambda \subset \Phi^{m}_{\lambda - \ord(a)}.
    \label{eq:multiplication order}
\end{equation}
We define the weighted Nash class (of germs)
\begin{equation}
    \Phi^{m,\cN}_\lambda := \cN((0,\varepsilon))\cap\Phi^m_\lambda,
\end{equation}
where, as usual, the interval may be shrunk without changing the associated germ at $0^+$.

As a direct consequence of the Puiseux theorem (Section~\ref{sect:semialgebraic limit}), we have the following. 

\begin{lemma}\label{lem:Nash derivative growth} 
Let $u\in\cN((0,\varepsilon))$ and let $\lambda\in\R$. Suppose \[ |u(t)|\leq Ct^\lambda
\text{\quad for all sufficiently small $t>0$.}
\]  
Then for $0 \leq r \leq m$, 
\begin{equation*}
    |u^{(r)}(t)| \leq C_rt^{\lambda-r}
    \text{\quad for all sufficiently small $t>0$.}
\end{equation*}
In particular, $u\in\Phi^{m,\cN}_\lambda$.
\end{lemma}

Recall Lemma~\ref{lem:weak lattice for arbitrary family} and Proposition~\ref{prop:nash weak lattice}.

\begin{lemma}\label{lem:weighted Nash lattice} 
Let $\lambda\in \mathbb Q$, $p,q \in \mathbb N_{\geq 1}$, and let $L_1,\ldots,L_p,U_1,\ldots,U_q \in\Phi^{m,\cN}_\lambda$. Suppose
\begin{equation}
    L_i(t)\leq U_j(t)
    \text{\quad for all $i,j$ and all sufficiently small $t>0$.}
\label{eq:weighted LiUj}
\end{equation}
Then there exists $h\in\Phi^{m,\cN}_\lambda$ such that
\begin{equation*}
    L_i(t)\leq h(t)\leq U_j(t)
    \text{\quad for all $i,j$ and all sufficiently small $t>0$.}
\end{equation*}
\end{lemma}

\begin{proof} 
Let $\eps > 0$ be sufficiently small so that the system of inequalities~\eqref{eq:weighted LiUj} hold for all $t \in (0,\varepsilon)$. Since $L_1,\ldots,L_p,U_1,\ldots,U_q \in\cN((0,\varepsilon))$, Proposition~\ref{prop:nash weak lattice} and Lemma~\ref{lem:weak lattice for arbitrary family} yield a function $h\in\cN((0,\varepsilon))$ such that 
\begin{equation*}
    L_i\leq h\leq U_j
    \text{\quad for all $i$ and $j$.}
\end{equation*}
Fix $i_0\in\{1,\ldots,p\}$ and $j_0\in\{1,\ldots,q\}$. Since $L_{i_0}(t)\leq h(t)\leq U_{j_0}(t)$ and $L_{i_0},U_{j_0}\in\Phi^{m,\cN}_\lambda$, there exists a constant $C>0$ such that
\begin{equation*}
    \abs{L_{i_0}(t)}+\abs{U_{j_0}(t)} \leq Ct^\lambda
    \text{\quad for all sufficiently small $t > 0$.}
\end{equation*}
It follows that 
\begin{equation*}
    \abs{h(t)}\leq Ct^\lambda
    \text{ for all $t \in (0,\eps)$.}
\end{equation*}
Lemma~\ref{lem:Nash derivative growth} then implies $h\in\Phi^{m,\cN}_\lambda$. 
This completes the proof. 
\end{proof}

The main result of this section is the following ``weighted box selection'' lemma. Intuitively, the lemma says that nonemptiness of a Nash family of bounded polyhedra implies the existence of a Nash selection whose vanishing order is controlled uniformly. The proof inducts on the number $M$ of variables, using Fourier-Motzkin elimination to first project and then lift the system while preserving the weighted Nash class. The additional coordinate is then constructed using Lemma~\ref{lem:weighted Nash lattice}.

\begin{lemma}\label{lem:weighted nash box selection}
    Let $(B_{ij}(t))_{1\leq i \leq L, 1 \leq j \leq M}$ be a matrix of right-hand Nash germs in $\germ_{0}$. There exists a constant $D_{\mathrm{box}} \geq 0$ depending only on $B$ and $M$ such that the following holds. 

    Let $\lambda,\Lambda \in \mathbb Q$ with $\Lambda \geq \lambda$. Suppose $g_i \in \Phi^{m,\cN}_\lambda$ for each $i = 1, \cdots, L$. Assume that for every sufficiently small $t > 0$,
    \begin{equation*}
        K(t) := \set{v \in \R^M:
        B(t)v \leq g(t)
        \text{ and }
        \abs{v_j}\leq C_0t^\Lambda\text{ for $j = 1, \cdots, M$}
        } \neq \void.
    \end{equation*}
    Then, after shrinking the interval if necessary, there exists $v = (v_1, \cdots, v_M) \in C^m((0,\eps), \R^M)$ such that $v(t)\in K(t)$ for all sufficiently small $t$, and
    \begin{equation*}
        v_j \in \Phi^{m,\cN}_{\lambda - D_{\mathrm{box}}} 
        \text{\quad for $j = 1, \cdots, M$.}
    \end{equation*}
\end{lemma}

\begin{proof}

    Since each $B_{ij}\in \germ_0$, each nonzero entry has a finite rational order, which may be negative. We first normalize the rows of the system. For each $i = 1, \cdots, L$, choose $\delta_i \in \mathbb Q_{\geq 0}$ such that
    \begin{equation*}
        \delta_i + \operatorname{ord}_{0}(B_{ij}) \geq 0
        \text{\quad for all $j$ such that $B_{ij}\not\equiv 0$}.
    \end{equation*}
    Multiplying the $i$-th row by the function $t\mapsto t^{\delta_i}$ (which is positive on $(0,\eps)$) preserves the solution. Moreover, by~\eqref{eq:multiplication order},
    \begin{equation*}
        t^{\delta_i}g_i \in \Phi^{m,\cN}_{\lambda+\delta_i} \subset \Phi^{m,\cN}_{\lambda}.
    \end{equation*}
    For the rest of the proof, we may then assume that 
    \begin{equation*}
        \operatorname{ord}_0(B_{ij}) \geq 0 
        \text{\quad for every $B_{ij}\not\equiv 0$.}
    \end{equation*}

    We then augment the system to incorporate the box constraints $\norm{v}_\infty \leq C_0t^\Lambda$. Define the block-matrix functions
    \begin{equation*}
        \widetilde B(t):= (B(t), I_M, -I_M)
        \text{\quad and \quad}
        \widetilde g(t) := (g(t), C_0t^\Lambda \mathbf 1_N, C_0t^\Lambda \mathbf 1_N)
    \end{equation*}
    so that
    \begin{equation}
        K(t) = \set{v \in \R^M : \widetilde B(t)v \leq \widetilde g(t)}.
        \label{eq:augmented system}
    \end{equation}

    Since $\Lambda \in \mathbb Q$, the function $t\mapsto t^\Lambda$ is Nash on a sufficiently small interval $(0,\eps_\Lambda)$. Since $\lambda \leq \Lambda$, we have $t^\Lambda \in \Phi^{m,\cN}_\lambda$. Therefore,
    \begin{equation*}
        \widetilde g_\alpha \in \Phi^{m,\cN}_\lambda \text{\quad for every row $\alpha$.}
    \end{equation*}

    We will prove the lemma by induction on $M$. 
    
    The case $M = 0$ is trivial. 
    
    In each induction step, we assume -- after shrinking the interval once if necessary -- that every nonzero coefficient from the Fourier-Motzkin eliminations has constant sign. Assume the lemma holds for $M-1$. 

    Write $v = (v', s) \in \R^{M-1}\times \R$ and write each row of the augmented system~\eqref{eq:augmented system} as
    \begin{equation*}
        a_\alpha(t) s + c_\alpha(t) \cdot w \leq r_\alpha(t)
        \text{\quad with \quad} r_\alpha \in \Phi^{m,\cN}_\lambda.
    \end{equation*}
    Note that each $a_\alpha$ has constant sign after shrinking the interval, so we may partition the indices as follows.
    \begin{equation*}
        \mathcal I_- := \set{\alpha: a_\alpha < 0},\ 
        \mathcal I_+ := \set{\alpha: a_\alpha > 0},\
        \mathcal I_0 := \set{\alpha:a_\alpha \equiv 0}.
    \end{equation*}
    Note that $\mathcal I_-$ and $\mathcal I_+$ are nonempty, since the boundary of the box $-C_0t^\Lambda \leq s \leq C_0t^\Lambda$ poses an active constraint. 

    We proceed as in the proof of Proposition~\ref{prop:solution on regular interval}.
    \begin{itemize}
        \item For $\alpha \in \mathcal I_-$, define 
        \begin{equation*}
            \eqindent
            L_\alpha^\sharp(t,v') := \frac{r_\alpha(t) - c_\alpha(t) \cdot v'}{a_\alpha(t)}.
        \end{equation*}
        \item For $\beta \in \mathcal I_+$, define
        \begin{equation*}
            \eqindent
            U_\beta^\sharp(t,v') := \frac{r_\beta(t)-c_\beta(t)\cdot v'}{a_\beta(t)}.
        \end{equation*}
        \item For $\gamma \in \mathcal I_0$, we have a row that is independent of $s$:
        \begin{equation*}
            \eqindent
            c_\gamma(t)\cdot v' \leq r_\gamma(t).
        \end{equation*}
    \end{itemize}
    Therefore, the Fourier-Motzkin projection of the system~\eqref{eq:augmented system} is given by
    \begin{equation}
        \begin{cases}
            L_\alpha^\sharp(t,v') \leq U_\beta^\sharp(t,v')
            \Leftrightarrow
            (a_\beta c_\alpha - a_\alpha c_\beta)(t)\cdot v' \leq (a_\beta r_\alpha - a_\alpha r_\beta)(t)
            &\text{ $\alpha \in \mathcal I_-$, $\beta \in \mathcal I_+$}
            \\
            c_\gamma(t)\cdot v' \leq r_\gamma(t)
            &\text{ $\gamma \in \mathcal I_0$}
        \end{cases} \ .
        \label{eq:augmented system projected}
    \end{equation}
    In the first set of inequalities in~\eqref{eq:augmented system projected}, all the coefficients are Nash germs of nonnegative order, so~\eqref{eq:multiplication order} implies $a_\beta r_\alpha - a_\alpha r_\beta \in \Phi^{m,\cN}_\lambda$. 

    By the induction hypothesis, the system~\eqref{eq:augmented system projected} admits a solution
    \begin{equation}
        v' = (v'_1, \cdots, v'_{M-1}) \in \bigoplus_{j=1}^{M-1} \Phi^{m,\cN}_{\lambda - D_{M-1}}
        \label{eq:inducted solution v prime}
    \end{equation}
    where $D_{M-1}$ depends only on the projected coefficients. 
    
    Define the univariate functions
    \begin{equation*}
        L_\alpha(t):= L_\alpha^\sharp(t,v'(t))
        \text{\quad and \quad}
        U_\beta(t) := U_\beta^\sharp(t,v'(t))
        \text{\quad for $\alpha \in \mathcal I_-$ and $\beta \in \mathcal I_+$.}
    \end{equation*}
    We analyze the orders of these functions.
    \begin{itemize}
        \item Since $D_{M-1} \geq 0$, we have $\Phi^{m,\cN}_\lambda \subset \Phi^{m,\cN}_{\lambda - D_{M-1}}$. 
        \item Multiplication by $c_\alpha$ preserves the class $\Phi^{m,\cN}_{\lambda - D_{M-1}}$, so $r_\alpha - c_\alpha \cdot v' \in \Phi^{m,\cN}_{\lambda - D_{M-1}}$. 
        \item Set
    \begin{equation*}
        \kappa := \max\brac{
        \set{\operatorname{ord}_0
        a_\alpha : \alpha \in \mathcal I_- \cup \mathcal I_+}
        \cup \set{0} }.
    \end{equation*}
    Then division by $a_\alpha$ lowers the order by at most $\kappa$.
    \end{itemize}
    In short, we have
    \begin{equation*}
        L_\alpha, U_\beta \in \Phi^{m,\cN}_{\lambda- D_{M-1} - \kappa}
        \text{\quad for all $\alpha \in \mathcal I_-$ and $\beta \in \mathcal I_+$.}
    \end{equation*}
    Since $v'$ from~\eqref{eq:inducted solution v prime} satisfies the Fourier-Motzkin projection, we have
    \begin{equation*}
        L_\alpha(t) \leq U_\beta(t)
        \text{\quad for all $\alpha \in \mathcal I_-$ and $\beta \in \mathcal I_+$.}
    \end{equation*}
    By Lemma~\ref{lem:weighted Nash lattice}, there exists
    \begin{equation}
        s \in \Phi^{m,\cN}_{\lambda- D_{M-1} - \kappa}
        \label{eq:s in weighted class}
    \end{equation}
    such that
    \begin{equation*}
        L_\alpha(t) \leq s(t) \leq U_\beta(t)
        \text{\quad for all $\alpha \in \mathcal I_-$ and $\beta \in \mathcal I_+$.}
    \end{equation*}
    
    To finish the proof, we simply define $v(t) := (v'(t),s(t))$. By construction, $v$ satisfies the augmented system~\eqref{eq:augmented system}, so $v(t)\in K(t)$ for all sufficiently small $t$. By~\eqref{eq:inducted solution v prime} and~\eqref{eq:s in weighted class}, every component of $v$ belongs to $\Phi^{m,\cN}_{\lambda - D_{\mathrm{box}}}$ with $D_{\mathrm{box}}:= D_{M-1} + \kappa$ and $\kappa$ depends only on $\widetilde B$ and $M$.

\end{proof}

\subsection{Flatness at an exceptional point}
\label{sect:extension at exceptional point}

The main goal of this section is captured in the following self-contained theorem, which characterizes the $m$-jets that arise from feasible one-sided $C^m$ selections at an exceptional point. 
The necessity follows directly from Taylor's theorem. 
For sufficiency, the key ingredients and main steps are as follows.
\begin{itemize}
    \item We first provide a quantitative upgrade for the distance condition from Proposition~\ref{prop:dist to polygon}(\ref{prop:dist to polygon point}) using Puiseux asymptotics.
    
    \item The semialgebraic curve selection lemma then provides an approximate feasible selection with the prescribed jet.

    \item The remaining constraint error is divided into asymptotically inactive and active rows, and the weighted box-selection lemma, ultimately based on the weighted Nash sandwich property, supplies a sufficiently flat correction for the active subsystem. Adding this correction produces the desired feasible $C^m$ selection without changing its $m$-jet.
\end{itemize}

The result is also constructive. See Remark~\ref{rem:curve-selection-constructive} after the proof.

\begin{theorem}\label{thm:one sided selection at exceptional points}
Fix $m \in \mathbb N_0$, $M,N \in \mathbb N_{> 0}$, and let $(A_{ij})_{1\leq i\leq N,\ 1\leq j\leq M}$ be a matrix of  univariate semialgebraic functions. Let $f=(f_1,\ldots,f_N) \in \cN(\mathbb R,\mathbb R^N)$. For each $x\in\mathbb R$, set
\begin{equation*}
    K_f(x)
:=
\set{v\in\mathbb R^M:A(x)v\leq f(x)}.
\end{equation*}
Let $\Sigma_A$ be the finite exceptional set associated with $A$ in the sense of Definition~\ref{def:regular interval exception point}, and fix $y\in\Sigma_A$.
Assume that a regular interval lies immediately to the right of $y$, and that for some $\eps>0$, $K_f(y+t)\neq \void$ for all $0 < t < \eps$. Let $P\in\mathcal P^m(\R,\R^M)$ Then the following are equivalent.
\begin{enumerate}[(A)]
\item There exist $\eps'\in \mathbb Q_{>0}$ and $F \in C^m([y,y+\eps'), \R^M)$
such that $F$ restricted to $(y,y+\eps')$ is Nash, $\jet_y^mF = P$, and $A(x)F(x)\leq f(x)$ for all $y \leq x < y+\eps'$.

\item One has $P(y)\in K_f(y)$ and 
\begin{equation*}
    d^+_{y,P}(t)
:=
\operatorname{dist}_\infty
\bigl(
P(y+t),K_f(y+t)
\bigr)
=
o(t^m)
\text{\quad as $t \downarrow 0$}
\end{equation*}
\end{enumerate}

The analogous statement holds from the left.
\end{theorem}

\begin{proof}
    We will prove the theorem for when the regular interval is on the right of $y$. The proof of the left-side is almost identical. 

    The implication (A) $\implies$ (B) is a direct consequence of Taylor's theorem for $C^m$. We will show that (B) $\implies$ (A). Without loss of generality, we may assume that $y = 0$.

    By Proposition~\ref{prop:dist to polygon}(\ref{prop:dist to polygon boolean}), there exists some rational number $\eta > 0$ such that
    \begin{equation}
        d_{0,P}^+(t) = O(t^{m+\eta})
        \text{ for sufficiently small $t > 0$.}
        \label{eq:d0P}
    \end{equation}
    By Taylor's theorem, for a sufficiently large integer $L$, we have
    \begin{equation}
        \abs{f_i(t) - \jet_{0}^{(L)}f_i(t)} \leq C_L t^{L+1}
        \text{\quad for all $i$ and sufficiently small $t > 0$.}
        \label{eq:taylor L}
    \end{equation}
    Fix the constant $C_L$ and define a relaxed Taylor data
    \begin{equation*}
        f^{[L]}(t):= \jet_0^{(L)}f(t) + C_L t^{L+1}\mathbf{1}_N.
    \end{equation*}
    It follows from \eqref{eq:taylor L} that $f(t)\leq f^{[L]}(t)$ coordinate-wise for small $t$ and $K_f(t) \subset K_{f^{[L]}}(t)$. By~\eqref{eq:d0P}, we have
    \begin{equation*}
        \dist_\infty(P(t), K_{f^{[L]}}(t)) \leq C_0t^{m+\eta}.
    \end{equation*}
    Fix $C_0$ as above and consider the semialgebraic set
    \begin{equation*}
        \mathcal R_0 := \set{(t,R)\in \R\times \R^M: 0 < t < \eps,\ A(t)[P(t) + R] \leq f^{[L]}(t), \ 
        \norm{R}_\infty \leq 2C_0t^{m+\eta}
        }
    \end{equation*}
    and we write $\mathcal R_0(t) = \set{R\in \R^M: (t,R) \in \mathbb R_0}$. After shrinking $\eps > 0$ if necessary, we may assume that every fiber $\mathcal R_0(t)$ is nonempty. By the semialgebraic curve selection lemma, there exists a semialgebraic function $R_0(t)$ such that
    \begin{equation}
        (t,R_0(t)) \in \mathcal R_0
        \text{\quad for all $t \in (0,\eps)$.}
        \label{eq:R0 curve selection}
    \end{equation}
    Shrinking $\eps > 0$ again if necessary, we may assume $R_0$ is smooth in $t$ and $\norm{R_0(t)}_\infty = O(t^{m+\eta})$. By Puiseux' theorem, every component of $R_0$ has a Puiseux order of at least $m+\eta$. Therefore,
    \begin{equation*}
        R_0^{(r)}(t) = O(t^{m+\eta - r})
        \text{\quad for all $0\leq r \leq m$ and $t \downarrow 0$.}
    \end{equation*}
    Therefore, $R_0$ extends to a $C^m$ function at $0$ with a zero $m$-jet. 

    Next, we define two functions that measure the solution gaps at $t$:
    \begin{equation*}
        \begin{split}
            g_1(t) &:= f^{[L]}(t) - A(t)[P(t) + R_0(t)],
            \\
            g_2(t) &:= f(t) - f^{[L]}(t).
        \end{split}
    \end{equation*}
    Note that $g_1(t)$ is component-wise nonnegative and $g_2(t)$ is component-wise nonpositive. We claim that it suffices to find an error function $E(t)$ satisfying
    \begin{equation}
        A(t)E(t) \leq g_1(t) + g_2(t)
        \text{\quad and \quad}
         \jet_0^{(m)}E\equiv 0.
        \label{eq:error Et}
    \end{equation}
    Indeed, once we establish~\eqref{eq:error Et}, we may set $R(t) := R_0(t) + E(t)$, and $F(t):= P(t) + R(t)$ is the desired solution. It remains to show~\eqref{eq:error Et}.

    We begin with the first part of~\eqref{eq:error Et}. From \eqref{eq:taylor L}, we see that 
    \begin{equation}
        \text{$g_2 \in \Phi^m_{L+1}$ \quad and \quad  $ \abs{g_{2}(t)} \leq C\abs{t}^{L+1}$.}
        \label{eq:g2 taylor}
    \end{equation}
    Therefore, 
    \begin{equation*}
        \norm{\brac{A(t)[P(t)+R_0(t)]
        - f(t)
        }_+} \leq Ct^{L+1}.
    \end{equation*}
    By Proposition~\ref{prop:dist to polygon}(\ref{prop:dist to polygon point}), for each sufficiently small $t$, there exists a point $v^*(t)\in K_f(t)$ such that
    \begin{equation*}
        \norm{v^*(t) - [P(t)+R_0(t)]}_\infty \leq C't^{L+1-D_{\mathrm{dist}}}
    \end{equation*}
    where $C'$ and $D$ are constants depending only on $A$. Using $v^*$, we define
    \begin{equation}
        E^*(t) := v^*(t) - [P(t)+R_0(t)].
        \label{eq:Estar}
    \end{equation}
    It follows from construction that $A(t)E^*(t) \leq g_1(t) + g_2(t)$, with bound
    \begin{equation}
        \norm{E^*(t)}_\infty \leq C't^\Lambda
        \text{\quad where $\Lambda = L+1-D_{\mathrm{dist}} \in \mathbb Q$.}
        \label{eq:Estar bounded by Lambda}
    \end{equation}
    Thus, for every sufficiently small $t>0$, 
    \LAQ{eq:system feasible}{the system 
    \begin{equation*}
        A(t)E\leq g_1(t)+g_2(t) \text{\quad with constraint \quad} \norm{E}_\infty\leq C't^\Lambda,
    \end{equation*}
    is nonempty, with $E^*(t)$ providing a pointwise feasible element.}
    Note that there is no regularity assumption on $E^*(t)$ for the moment. We will use Lemma~\ref{lem:weighted nash box selection} to obtain a $C^m$ choice of such elements.

    To this end, set
    \begin{equation*}
        \alpha_i := \min_{j : A_{ij}\not\equiv 0}\ord_0(A_{ij}).
    \end{equation*}
    Note that we only define $\alpha_i$ when the entire row is not identically zero. If the $i$-th row of $A$ is identically zero, then the corresponding inequality is independent of $E$. Recall from~\eqref{eq:system feasible} that pointwise feasible element $E^*(t)$ satisfies the full system, we have 
    \begin{equation*}
        0\leq g_{1,i}(t)+g_{2,i}(t)
    \end{equation*}
    for all sufficiently small $t>0$. Therefore, the identically zero rows impose no condition on the choice of $E$ and we omit them from the arguments below.
    
    For the nonzero rows, also set
    \begin{equation*}
        \beta_i := \begin{cases}
            \ord_0(g_{1,i}) &\text{ if $g_{1,i}\not\equiv 0$}
            \\
            +\infty &\text{ otherwise}
        \end{cases}\ .
    \end{equation*}
    Choose an integer $L_0$ such that \begin{equation} 
    L_0> \max\brac{\{D_{\mathrm{dist}}-1-\alpha_i: A_i\not\equiv0\} \cup\{D_{\mathrm{dist}}-1\}}. 
    \label{eq:L0 def} 
    \end{equation}
    For each row $i$, we will separately treat the cases
    \begin{equation*}
        \beta_i < L - L_0
        \text{\quad and \quad}
        \beta_i \geq L - L_0.
    \end{equation*}
    
    \begin{itemize}
        \item Suppose $\beta_i < L-L_0$. We will show that these rows pose no asymptotic constraints. 
        Since $g_{1,i} \geq 0$, Puiseux' theorem states that the local behavior of $g_{1,i}$ is governed by $c_it^{\beta_i}$ for some $c_i > 0$. Thanks to~\eqref{eq:L0 def}, we have
        \begin{equation*}
        \eqindent
            A_i(t)E = o(t^{L-L_0})
            \text{ whenever }
            \norm{E}_\infty \leq C't^\Lambda.
        \end{equation*}
        Thanks to~\eqref{eq:g2 taylor}, we have
        \begin{equation*}
        \eqindent
            g_{2,i}(t) = O(t^{L+1}) = o(t^{L-L_0}) = o(t^{\beta_i}),
        \end{equation*}
        where we use the assumption that $\beta_i < L-L_0$ in the last asymptotic identity. Combining all of the above, we see that
        \begin{equation*}
        \eqindent
            \abs{A_i(t)E} + \abs{g_{2,i}(t)} = o(g_{1,i}(t))
            \text{\quad for $\norm{E}_\infty \leq C't^\Lambda$ with $t > 0$ sufficiently small.}
        \end{equation*}
        In particular, \eqref{eq:error Et} is automatic for these rows.

        \item Suppose $\beta_i \geq L-L_0$. If $g_{1,i}\not\equiv 0$, then Puiseux' theorem gives $g_{1,i}\in \Phi^m_{\beta_i} \subset \Phi^m_{L-L_0}$. If $g_{1,i}\equiv 0$, then $g_{1,i}\in\Phi^m_{L-L_0}$ trivially. Since $A$, $f$, $f^{[L]}$, and $R_0$ are Nash on the punctured regular interval $(0,\eps)$, the functions $g_1$ and $g_2$ are Nash there. Together with~\eqref{eq:g2 taylor}, this gives \[ g_{1,i}+g_{2,i} \in\Phi^{m,\cN}_{L-L_0} \] for every remaining row.
        
        Consider the subsystem formed by these remaining rows coupled with the constraint $\norm{E}_\infty \leq C't^\Lambda$. By~\eqref{eq:system feasible}, the subsystem is also pointwise feasible. By Lemma~\ref{lem:weighted nash box selection} applied to this subsystem with $\lambda = L - L_0$, there exists a $C^m$ solution $E(t)$ to the subsystem such that
        \begin{equation*}
            E_j \in \Phi^{m,\cN}_{L-L_0 - D}
        \end{equation*}
        for some $D$ depending only on $(A_{ij}),M,m$. 
    \end{itemize}
    
    Choose $L$ so large that $L-L_0 - D > m$, so that every component of $E$ extends to a $C^m$ function with $\jet_0^{(m)}E \equiv 0$. Since the case $\beta_i < L-L_0$ imposes no asymptotic constraint, we see that $E$ satisfies the full system~\eqref{eq:error Et}.

    This concludes the proof of Theorem~\ref{thm:one sided selection at exceptional points}. 
    
\end{proof}

\begin{remark}\label{rem:curve-selection-constructive}
    The proof of Theorem~\ref{thm:one sided selection at exceptional points} is constructive in the semialgebraic sense. In particular, the (parametrized univariate representation together with a Thom encoding of the) auxiliary curve $R_0$ from~\eqref{eq:R0 curve selection} may be computed using an effective semialgebraic curve-selection algorithm; see, e.g.~\cite{basu2022curve-select}.  
\end{remark}

Finally, we formulate the condition in Theorem~\ref{thm:one sided selection at exceptional points} in Boolean. 

\begin{proposition}\label{prop:boolean exceptional points}
    The following are true about exceptional points.
    \begin{enumerate}[(A)]
    \item For every $y\in\Sigma_A$ and each side $\eps \in \set{-,+}$, there is an integer $s_y$ and a finite Boolean formula $\mathcal B^\eps_y\brac{P,\jet_y^{(s_y)}f}$ built from homogeneous linear sign conditions such that, assuming pointwise feasibility on the corresponding punctured regular interval, $P\in\mathcal P^m(\mathbb R,\mathbb R^M)$ is realized by a feasible one-sided $C^m$ selection if and only if $\mathcal B^\eps_y\brac{P,\jet_y^{(s_y)}f}$ holds.
    

    \item For every $y \in \Sigma_A$, the condition
    \begin{equation}
        \eqindent
\brac{\exists P\in\mathcal P^m(\mathbb R,\mathbb R^M)}
\,
\left[
A(y)P(y)\leq f(y)
\ \wedge\
\mathcal B^-_y
\bigl(
P,\jet_y^{(s_y)}f
\bigr)
\ \wedge\
\mathcal B^+_y
\bigl(
P,\jet_y^{(s_y)}f
\bigr)
\right]
\label{eq:end point feasibility exceptional point}
    \end{equation}
    is equivalent to a finite Boolean combination of homogeneous linear sign conditions on a finite jet $\jet_y^{(s_y)}$, with coefficients purely determined by $(A_{ij}),M,N$, and $m$.

\end{enumerate}
\end{proposition}

\begin{proof}
    We begin with (A).
    By Theorem~\ref{thm:one sided selection at exceptional points}, the one-sided admissibility is equivalent to 
    \begin{equation*}
        A(y)P(y) \leq f(y)
        \text{\quad and \quad}
        d_{y,P}^\pm(t) = o(t^m).
    \end{equation*}
    The first condition is a conjunction of homogeneous linear inequalities in $P$ and $f(y)$, and the second is a combination of finite Boolean-linear sign conditions in $(P,\jet_y^{(s_y)}f)$ by Proposition~\ref{prop:dist to polygon}(\ref{prop:dist to polygon boolean}).

    Now we turn to (B). Note that \eqref{eq:end point feasibility exceptional point} has the form $(\exists w \in W) \ \mathcal B(u,w)$ where $U\ni u, W$ are Euclidean. Put $\mathcal B$ in disjunctive normal form and eliminate $w$ coordinate-wise. Each step of the elimination produces a finite Boolean combination of homogeneous linear sign conditions.
\end{proof}

\section{Proof of Theorem~\ref{thm:main}}
\label{sect:proof of main theorem}

We begin by constructing the formulas in the statement of~Theorem \ref{thm:main}.

Recall $\Sigma_A$ from Definition~\ref{def:regular interval exception point}.
    Refine $\Sigma_A$ if necessary, by the finitely  many zero-dimensional cells occurring in the Fourier-Motzkin description of pointwise feasibility. Let $\mathcal C$ be the finite semialgebraic partition of $\R$ consisting of regular intervals and exceptional points.

    On each regular interval $I \subset \R \setminus \Sigma_A$, Proposition~\ref{prop:Fourier Motzkin} produces a finite conjunction $\mathcal F_I(f(x))$ of homogeneous linear inequalities with semialgebraic coefficients on $I$ depending only on $(A_{ij})$,
    such that 
    \begin{equation}
        \text{$K_f(x) = \set{v\in \R^M: A(x)v \leq f(x)} \neq \void$ if and only if $\mathcal F_I(f(x))$.}
        \label{eq:Kf and FI}
    \end{equation}

    For each exceptional point $y \in \Sigma_A$, Proposition~\ref{prop:boolean exceptional points} yields a finite Boolean formula $\mathcal C_y(\jet_y^{(s_y)}f)$ in linear homogeneous sign conditions, which is equivalent to the feasibility of a jet at $y$. In view of this, we define
    \begin{equation*}
        \bar m := \max\set{m \cup \set{s_y : y \in \Sigma_A}}.
    \end{equation*}
    Since $\Sigma_A$ is finite, the maximum is well defined and finite. Adding zeros if necessary, all formulae can be regarded as sign formulas in $\jet_y^{\bar m}f$.

    Finally, we set
    \begin{equation*}
        \mathcal S_I = \begin{cases}
            \mathcal F_I &\text{ if $I\subset \R\setminus \Sigma_A$ is a regular interval}
            \\
            \mathcal C_y &\text{ if $I = \set{y} \subset \Sigma_A$ is an exceptional point}
        \end{cases} \ .
    \end{equation*}

\begin{proof}[Proof of necessity]

Suppose $F\in C^m(\R,\R^M)$ satisfies $A(x)F(x)\leq f(x)$ for every $x \in \R$. 

\begin{itemize}
    \item If $x \in I$ where $I\subset \R\setminus \Sigma_A$ is a regular interval, then $F(x) \in K_f(x)$ and $K_f(x) \neq \void$. By~\eqref{eq:Kf and FI}, $\mathcal F_I(f(x))$ is true.

    \item Suppose $y \in \Sigma_A$ is an exceptional point. Set $P_y := \jet_y^{(m)} F \in \mathcal P(\R,\R^M)$. The restriction of $F$ to either side of $y$ is a feasible one-sided $C^m$ selection with the same jet $P_y$. By Proposition~\ref{prop:boolean exceptional points}(A), $P_y$ satisfies both one-sided endpoints formulae. Moreover, $P_y(y) = F(y)\in K_f(y)$, and by Proposition~\ref{prop:boolean exceptional points}(B), we see that $\mathcal C_y(\jet_y^{(s_y)}f)$ holds. 
\end{itemize}
This proves necessity.     
\end{proof}

\begin{proof}[Proof of sufficiency]
    Assume that all of the cell-wise conditions hold. 
    \begin{itemize}
        \item On every regular interval $I\subset \R\setminus \Sigma_A$, the condition $\mathcal F_I(f(x))$ holds for all $x \in I$. By~\eqref{eq:Kf and FI}, $K_f(x)\neq \void$ for every $x \in I$. By Proposition~\ref{prop:solution on regular interval}, there exists $G_I \in \cN(I,\R^M)\subset C^m(I,\R^M)$ such that $G_I(x) \in K_f(x)$ for all $x \in I$.

        \item For every exceptional point $y \in \Sigma_A$, the formula $\mathcal C_y(\jet_y^{(s_y)}f)$ holds. By Proposition~\ref{prop:boolean exceptional points}(B), there exists an $m$-jet $P_y \in \mathcal P^m(\R,\R^M) \in \cN(\R,\R^M)$ realizing~\eqref{eq:end point feasibility exceptional point} and $P_y(y) \in K_f(y)$. By Theorem~\ref{thm:one sided selection at exceptional points}, there exist $\eps_y > 0$ and feasible one-sided $C^m$ selection $F_y^-$ and $F_y^+$ defined on $(y - \eps_y, y]$ and $[y, y+\eps_y)$, respectively, such that
        \begin{equation*}
            \jet_y^{(m)} F_y^- = P_y = \jet_y^{(m)}F_y^+.
        \end{equation*}
    \end{itemize}
    It remains to patch the two types of solutions together.
    
    Let $I = (a,b)$ be a regular interval. We assume that $a,b \neq \pm \infty$. The case that either one is infinite can be handled similarly. Fix $F_a^+$ and $F_b^-$ as above. Shrinking $\eps_a$ or $\eps_b$ if necessary, we assume that $[a,a+\eps_a) \cap (b-\eps_b,b] = \void$.

    Let $0 \leq \chi_a,\chi_b \leq 1$ be $C^m$ semialgebraic cutoff functions such that $\chi_\circ \equiv 1$ near $\circ \in \set{a,b}$. Modify their supports if necessary, we may assume that $\chi_a + \chi_b \leq 1$. Define $\chi_I := 1 - \chi_a - \chi_b$ and define
    \begin{equation*}
        F_I := \chi_aF_a^+ + \chi_I G_I + \chi_bF_b^-.
    \end{equation*}
    It is clear that $F_I \in C^m(I,\R^M)$ and semialgebraic
    and $F_I(x)\in K_f(x)$ for all $x \in I$ by convexity. Moreover, $F_I \equiv F_a^+$ near $a$ and $F_I \equiv F_b^-$ near $b$.

    Finally, define the global solution
    \begin{equation*}
        F(y) := \begin{cases}
            F_I(y) &\text{ if $y \in I$ and $I\subset \R\setminus \Sigma_A$ is a regular interval}
            \\
            P_y(y) &\text{ if $y \in \Sigma_A$.}
        \end{cases} \ .
    \end{equation*}
    At exceptional points, we have $\jet_y^{(m)}F_y^- = P_y = \jet_y^{(m)}F_y^+$. It follows that $F \in C^m(\R,\R^M)$. 
    Since each component is semialgebraic, $F$ is semialgebraic.
    It is also clear from construction that $F(y) \in K_f(y)$ for all $y \in \R$,  This concludes the proof.
\end{proof}

\section{Semimodules of admissible forcing vectors}
\label{sect:generators}

Theorem~\ref{thm:main} gives a finite differential sign description of the admissible forcings. We now ask for a complementary description by algebraic generators.

For the rest of this section, we fix $m \in \mathbb N_0$ and let $I\subset \R$ be an open and semialgebraic interval.
Let $\ring(I)$ denote a fixed ring of scalar functions on~$I$, such as $\R[x]$ restricted to $I$ or $\cN(I)$. If $I = \R$, we simply write $\ring$ or $\cN$.
Let $\sa^m(X,Y)$ denote the class of semialgebraic $C^m$ maps from a semialgebraic set $X$ to a semialgebraic set $Y$. 

We define the cone of admissible forcings by 
\begin{equation} 
\cone_A^m(\ring(I)) := \set{ 
f\in\ring(I)^N: \text{there exists } 
F\in\sa^m(I,\R^M) 
\text{ such that } 
AF\leq f 
\text{ on $I$}
}. 
\label{eq:semimodule main} 
\end{equation} 
By Theorem~\ref{thm:main}, the same cone with $I = \R$ is obtained if $F\in\sa^m(\R,\R^M)$ in~\eqref{eq:semimodule main} is replaced by $F\in C^m(\R,\R^M)$. 

Let $I$ be a regular interval in the sense of Definition~\ref{def:regular interval exception point}. The inequality $AF\leq f$ on $I$ is equivalent to the existence of a component-wise nonnegative semialgebraic slack variable $\sigma$ satisfying 
\begin{equation} 
AF+\sigma=f 
\text{\quad with } \sigma \in\sa^m(I,\R_+^N), 
\label{eq:slack}
\end{equation} 
where $\R_+^N:=[0,\infty)^N$.
Consequently, 
\begin{equation} 
\cone_A^m(\ring) = \left[ 
A\cdot\sa^m(I,\R^M) + \sa^m(I,\R_+^N) 
\right] 
\cap\ring(I)^N. 
\label{eq:slack representation} 
\end{equation} 
Define the semiring 
 \begin{equation*} 
 \ring_+(I) := \set{ 
 q\in\ring(I): 
 q\geq0\text{ on } I 
 }. 
 \end{equation*} 
 The cone $\cone_A^m(\ring(I))$ is closed under addition and satisfies 
 \begin{equation*} 
 qf\in\cone_A^m(\ring(I)) 
 \qquad \text{for every } q\in\ring_+(I) \text{ and } f\in\cone_A^m(\ring(I)), 
 \end{equation*} 
 since $A(qF) = qAF \leq qf$ whenever $f \in \cone_A^m(\R(I))$. Therefore, $\cone_A^m(\ring(I))$ has the structure of a semimodule over the semiring $\ring_+(I)$.

 In addition, we define the lineality space of $\cone_A^m(\ring(I))$ by 
 \begin{equation} 
 \lin_A^m(\ring(I))
 := 
 \cone_A^m(\ring(I)) \cap \brac{-\cone_A^m(\ring(I))}. \label{eq:lineality space} 
 \end{equation} 
Note that $\lin_A^m(\ring(I))$ is an $\ring(I)$-module.

We use $\cone_A^m(\ring(I))\big/\lin_A^m(\ring(I))$ to denote the image of $\cone_A^m(\ring(I))$ under the map
\begin{equation*}
    \ring(I)^N \to \ring(I)^N \Big / \lin_A^m(\ring(I)).
\end{equation*}
Note that $\cone_A^m(\ring(I))\big/\lin_A^m(\ring(I))$ is a semimodule over the semiring $\ring_+(I)$.

\subsection{Finite generations or lack thereof}

 \begin{proposition}\label{prop:lineality module} 
 Let $\ring(I)$ be $\R[x]$ restricted to $I$ or $\cN(I)$. Then
  $\lin_A^m(\ring(I))$ is an $\ring(I)$-submodule of $\ring(I)^N$. Consequently, $\lin_A^m(\ring(I))$ is finitely generated as an $\ring(I)$-module. 
\end{proposition}

\begin{proof}

Since $\cone_A^m(\ring(I))$ is a convex cone, $\lin_A^m(\ring(I))$ is an additive subgroup of $\ring^N(I)$. We claim that it is stable under multiplication by arbitrary elements of $\ring(I)$. 

First of all, $\lin_A^m(\ring(I))$ is stable under multiplication by $\ring_+(I)$, i.e., for every $q\in\ring_+(I)$, 
\begin{equation} 
qf\in\cone_A^m(\ring(I)) 
\text{\quad and\quad }
-qf=q(-f)\in
\cone_A^m(\ring(I)) \quad \implies \quad qf \in \lin_A^m(\ring(I)).
\label{eq:qf in lin}
\end{equation} 

Let $r\in\ring(I)$ be arbitrary. We make use of the elementary identity 
\begin{equation} 
r = \brac{\frac{r+1}{2}}^2 - \brac{\frac{r-1}{2}}^2 =: r_+ - r_-
\label{eq:difference of nonnegative elements} 
\end{equation} 
so that $r_\pm \in \ring_+(I)$. If $f\in\lin_A^m(\ring(I))$, by \eqref{eq:qf in lin}, $r_\pm f \in \lin_A^m(\ring(I))$. 
Since $\lin_A^m(\ring(I))$ is an additive subgroup, it follows that 
\begin{equation*} 
rf = r_+f-r_-f \in \lin_A^m(\ring(I)). 
\end{equation*} 
Therefore, $\lin_A^m(\ring(I))$ is an $\ring(I)$-submodule of $\ring^N(I)$. It remains to prove finite generation. 

\begin{itemize}
    \item Suppose $\ring(I)=\R[x]|_I$. Since $\R[x]$ is Noetherian, so is $\ring(I)$. 

    \item Suppose $\ring(I)=\cN(I)$. Since $I$ is open and semialgebraic, it is a Nash manifold, so $\cN(I)$ is Noetherian by~\cite[Theorem 8.7.18]{bookRAG}. 
\end{itemize}

For both cases, the free module $\ring^N(I)$ is a Noetherian $\ring(I)$-module, and every $\ring(I)$-submodule of $\ring^N(I)$ is finitely generated. In particular, $\lin_A^m(\ring (I))$ is finitely generated. This completes the proof.

\end{proof}

\begin{proof}[Proof of Theorem~\ref{thm:finite generation regular interval}]
    \renewcommand{\ring}{\cN}
    
    Let $f\in\cone_{A,I}^m(\ring(I))$.
    By Proposition~\ref{prop:solution on regular interval}, there exists 
    $F\in\cN(I,\R^M) $
    such that $AF\leq f$ on $I$. Let $\sigma:= f - AF$, so that $\sigma\in\ring_+(I)^N $. The vector $AF$ belongs to $\lin_{A,I}^m(\ring(I))$, since both $AF$ and $-AF=A(-F)$ belong to $\cone_{A,I}^m(\ring(I))$. Therefore, $[f]=[\sigma]$ as equivalent classes in $\ring(I)^N \big/ \lin_{A,I}^m(\ring(I))$. We can write $\sigma = \sum_{i=1}^N \sigma_ie_i$ where each $\sigma_i \in \ring_+(I)$, so that $[f] = \sum_i \sigma_i[e_i]$. This proves the theorem. 
\end{proof}

\medskip

 However, the situation becomes trickier when we deal with rings of polynomials as illustrated in the next theorem. This forms a sharp contrast with the main result from~\cite{FeffermanLuli2021Generators,FeffermanLuli2021Equations}.

 \begin{theorem}
 \label{thm:quotient module polynomial not FG} 
 Let $I:= (-\frac{1}{2},\frac{1}{2})$, $a(x) := \sqrt{1+x}$, and $A = (1,-a(x))$, so that $I$ is an open semialgebraic subinterval of a regular interval for $A$. Let $\ring(I)$ denote the ring of $\R[x]$ restricted to $I$. Then
 \begin{equation*}
     \lin_A^m(\ring (I)) = \set{0}
     \label{eq:trivial lineality module}
 \end{equation*}
 and $\cone_A^m(\ring(I))\big/\lin_A^m(\ring(I)) \cong \cone_A^m(\ring(I))$ is not finitely generated as a semimodule over $\ring_+(I)$.
 \end{theorem}

\begin{proof}

We write elements of $\cone_A^m(\ring(I))$ as $(p,q)$, so the system of inequalities in question becomes
$A(x)F(x) 
\leq (p(x),q(x))$
which is the same as
\begin{equation*} 
F(x)\leq p(x) \qquad\text{and}\qquad -a(x)F(x)\leq q(x). 
\end{equation*} 
Since $a(x)>0$ on $I$, Proposition~\ref{prop:nash weak lattice} implies that a solution $F$ exists if and only if $a(x)p(x)+q(x)\geq0 $ for all $x \in I$.
Consequently, 
\begin{equation} 
\cone_{A,I}^m(\ring(I)) 
= \set{
(p,q)\in\ring(I)^2: ap+q\geq0\text{ on }I 
}. 
\label{eq:explicit polynomial cone} 
\end{equation} 

We first determine the lineality space. If 
\begin{equation*}
    (p,q)\in\lin_{A,I}^m(\ring(I)) = \cone_A^m(\ring(I)) \cap (-\cone_A^m(\ring(I)))
\end{equation*}
then the feasibility condition in~\eqref{eq:explicit polynomial cone} forces
\begin{equation} 
ap+q\equiv 0 \qquad\text{ on }I. 
\label{eq:lineality rational relation} 
\end{equation} 
We claim that this forces $p=q=0$. Suppose otherwise, that $p\not\equiv0$, then \eqref{eq:lineality rational relation} gives $a = -\frac{q}{p}$ on every subinterval of $I$ on which $p$ is nonzero. Thus $a$ would be a rational function, which is a contradiction. Therefore $p\equiv0$, and \eqref{eq:lineality rational relation} then implies $q\equiv0$. Therefore,~\eqref{eq:trivial lineality module} holds, and
\begin{equation*}
    \cone_A^m(\ring(I)) \cong \cone_A^m(\ring(I))\Big / \lin_A^m(\ring(I))
\end{equation*}
where the identification is canonical.

Now, suppose toward a contradiction, that $\cone_A^m(\ring(I)) $ is finitely generated as a semimodule over $\ring_+(I)$, which means that there exist finitely many vectors 
\begin{equation*} 
w_\nu = (p_\nu,q_\nu)
\in
\cone_{A,I}^m(\ring(I)) \quad , \quad \nu=1,\ldots,s, 
\end{equation*} 
such that 
\begin{equation} \cone_{A,I}^m(\ring(I)) = \sum_{\nu=1}^s\ring_+(I)w_\nu. 
\label{eq:assumed polynomial finite generation} 
\end{equation}
The goal is to again find suitable polynomials $p$ and $q$ such that $a = \frac{q}{p}$.

Fix $x_0\in I$. 
We define a linear functional 
\begin{equation*} 
\ell_{x_0}(u,v) := a(x_0)u+v 
\end{equation*} 
and an associated closed half-space 
\begin{equation*} 
H_{x_0} 
:= 
\set{ 
(u,v)\in\R^2: \ell_{x_0}(u,v)\geq0 }. 
\end{equation*} 
We claim that 
\begin{equation} 
\set{ (p(x_0),q(x_0)): (p,q)\in\cone_{A,I}^m(\ring(I)) } = H_{x_0}. 
\label{eq:evaluation onto halfspace} 
\end{equation}

The $\subseteq$ direction follows immediately from \eqref{eq:explicit polynomial cone}. 

To prove $\supseteq$, take $(u,v)\in H_{x_0}$, 
set $p(x)\equiv u$, and set $q(x):= v - ua'(x_0)(x-x_0) + K(x-x_0)^2$ where $K \geq \frac{\abs{u}\norm{a''}_{L^\infty(I)}}{2}$. Then
\begin{equation*}
    \begin{split}
        a(x)p(x) + q(x) &= \big[a(x_0)u + v\big] + u\big[
        a(x) - a(x_0) - a'(x_0)(x-x_0)
        \big] + K(x-x_0)^2
        \\
        &\geq \big[a(x_0)u + v\big] + \brac{K - \abs{u}\frac{\norm{a''}_{L^\infty(I)}}{2}}(x-x_0)^2 \geq 0 \text{\quad for all $x \in I$.}
    \end{split}
\end{equation*}
This shows the reversed inclusion.

Consider the identity \eqref{eq:assumed polynomial finite generation} at $x_0$. We see that
\begin{equation} 
H_{x_0} = \ring_+(I)\text{-}\operatorname{cone} 
\set{ w_1(x_0),\ldots,w_s(x_0) }. 
\label{eq:pointwise finitely generated halfspace} 
\end{equation} 
with boundary
\begin{equation*} 
\partial H_{x_0} = \set{ (u,v)\in\R^2: a(x_0)u+v=0 }. \end{equation*} 
We claim that 
\begin{equation}
    \text{there exists $\nu$ with $w_\nu(x_0)\neq 0$ and $w_\nu(x_0)\in \partial H_{x_0}$.}
    \label{eq:wnu nonzero} 
\end{equation}
To see this, choose any nonzero vector $z\in\partial H_{x_0}$.
By \eqref{eq:pointwise finitely generated halfspace}, there exist constants $\gamma_\nu\geq0$ such that 
\begin{equation*} 
z = \sum_{\nu=1}^s\gamma_\nu w_\nu(x_0)
\text{\quad and }
\ell_{x_0}(z) = \sum_{\nu=1}^s \gamma_\nu \ell_{x_0}\bigl(w_\nu(x_0)\bigr) = 0.
\end{equation*} 
Since each $w_\nu$ belongs to the admissible cone, every term on the right-hand side is nonnegative, so $\ell_{x_0}\brac{w_{\nu}(x_0)} = 0$
whenever $\gamma_\nu>0$. Since $z\neq0$, \eqref{eq:wnu nonzero} holds.

For $\nu=1,\ldots,s$, define semialgebraic subsets
\begin{equation*} 
Z_\nu := \set{ 
x\in I: a(x)p_\nu(x)+q_\nu(x)=0 
\text{ and } w_\nu(x)\neq0 
}. \end{equation*} 
The preceding argument shows that 
\begin{equation*} 
I = \bigcup_{\nu=1}^s 
Z_\nu.
\label{eq:finite boundary cover} 
\end{equation*} 
The union is finite, so there exist an index $\nu_0$ and a nonempty open subinterval $J\subset I$ such that 
$J\subseteq Z_{\nu_0}$.
Therefore,
\begin{equation} 
a(x)p_{\nu_0}(x)+q_{\nu_0}(x)=0 
\qquad \text{for every }
x\in J. 
\label{eq:generator on moving boundary} 
\end{equation} 
Since $w_{\nu_0}(x)\neq0$ on $J$, the two polynomials $p_{\nu_0}$ and $q_{\nu_0}$ do not both vanish identically. If $p_{\nu_0}\equiv0$, then \eqref{eq:generator on moving boundary} would imply that $q_{\nu_0}$ vanishes on $J$ and hence is identically zero, a contradiction. 
Therefore, $p_{\nu_0}\not\equiv 0$. It follows from \eqref{eq:generator on moving boundary} that \begin{equation*} 
a(x) = -\frac{q_{\nu_0}(x)}{p_{\nu_0}(x)} 
\text{\quad is rational on a nonempty open subinterval of $J$}
\end{equation*} 
which is a contradiction.

As a consequence, $\cone_{A,I}^m(\ring(I))\cong \cone_A^m(\ring(I))\Big / \lin_A^m(\ring(I))$ is not finitely generated as a semimodule over $\ring_+(I)$. 
\end{proof}

\subsection{Obstruction in higher dimension}

The preceding finite-generation result is inherently one-dimensional. It was already shown in~\cite{LuliONeill2023Inequalities} that a finite characterization for solvability of~\eqref{eq:system main} is not always possible. Moreover, in two variables, finite generation can fail even for a positivity cone with an elementary finite description. 

We demonstrate this by the following elementary example from

\begin{example}\label{example:qjet}
\newcommand{\psd}{\operatorname{PSD}}
Let $\ring:= \R[x,y]$ and $\ring_+:= \set{ Q\in\ring: Q\geq0\text{ on }\R^2 }$. 
Consider the $\ring_+$-semimodule \begin{equation*} \cone := \set{ P\in\R[x,y]: P\geq0\text{ on }\R^2 \text{ and } P(0,0)=0 }. 
\end{equation*} 
Although $\cone$ has this elementary description, it is not finitely generated as a semimodule over $\ring_+$. To see this, let $\dot\jet^2_0 P$ denote the homogeneous quadratic part of $P$ at the origin. 

Since every $P\in\cone$ has a local minimum at the origin, its constant and linear terms vanish, and $\dot\jet^2_0P$ is a positive semidefinite quadratic form. Conversely, every positive semidefinite quadratic form belongs to $\cone$. Therefore, 
\begin{equation} 
\set{ \dot\jet^2_0 P: P\in\cone } = \set{ ax^2+2bxy+cy^2: \begin{pmatrix} a&b\\ b&c \end{pmatrix} 
\in \operatorname{PSD}(\R^2,\R^2) }. \label{eq:quadratic jets positive cone} \end{equation}

Suppose toward a contradiction that 
\begin{equation*} 
\cone = \sum_{\nu=1}^s\ring_+P_\nu 
\text{\quad for some $P_1,\ldots,P_s\in\cone$. }
\end{equation*} 
Let $Q\in\ring_+$. If $Q(0,0)>0$, then \begin{equation*} 
\dot\jet^2_0 (QP_\nu) = Q(0,0)\dot\jet^2_0P_\nu. \end{equation*} 
If $Q(0,0)=0$, then the nonnegativity of $Q$ implies that its constant and linear terms vanish. Since the same is true of $P_\nu$, the product $Q\cdot P_\nu$ vanishes to order at least four, so $\dot\jet^2_0(QP_\nu)\equiv 0$.
It follows that 
\begin{equation} 
\set{ \dot\jet^2_0P: P\in\cone } = \operatorname{cone} \set{ \dot\jet^2_0P_1,\ldots,\dot\jet^2_0P_s }. 
\label{eq:finite quadratic jet cone} \end{equation} 
On the other hand, for every $t\in\R$, the polynomial 
\begin{equation*} 
P_t(x,y):=(x+ty)^2 \in \cone
\end{equation*} 
and its quadratic form is given by the matrix $\begin{pmatrix} 1&t\\ t&t^2 \end{pmatrix} = (1,t)(1,t)^\top$.
Each such matrix spans an extreme ray of cone of $2\times2$ positive semidefinite matrices $\psd(\R^2)$, and distinct values of $t$ determine distinct extreme rays. Hence $\psd(\R^2)$ has infinitely many extreme rays and is not a finitely generated polyhedral cone, contradicting~\eqref{eq:finite quadratic jet cone}. Therefore, $\cone$ is not finitely generated over $\ring_+$.

\end{example}

\subsection{A subsequent algebraic question}

To conclude this paper, we pose a question and a remark. 

\begin{question}
\label{quest:effective algebraic generators} 
Let $A$ be a matrix of univariate semialgebraic functions, let $I\subset\R\setminus\Sigma_A$ be a regular interval in the sense of Definition~\ref{def:regular interval exception point}, and let $m\in\N_0$. 
\begin{enumerate}[(A)] 
\item Let the base ring $\ring(I) := \cN(I)$. Can a finite set of generators for the lineality module $\lin_{A,I}^m(\cN(I)) $
be computed effectively from $A$ and $m$? Can one also effectively compute the relations among the coordinate classes \begin{equation*} [e_1],\ldots,[e_N] \in \cN(I)^N\Big/\lin_{A,I}^m(\cN(I)), \end{equation*} 
so that one can obtain a minimal set of generators for $\cone_{A,I}^m(\cN(I)) \Big/ \lin_{A,I}^m(\cN(I))$ as a semimodule over $\cN_+(I)$?

\item Let the base ring $\ring(I):=\R[x]|_I$, and let $A \in \R[x]^{N\times M}$ be a matrix of univariate polynomials. Is $\cone_{A,I}^m(\ring(I)) \Big/ \lin_{A,I}^m(\ring(I))$ finitely generated as a semimodule over $\ring_+(I)$? More generally, under what additional algebraic conditions does such finite generation hold, and can a finite set of generators be computed effectively from $A$ and $m$? 
\end{enumerate} 
\end{question}

As mentioned in the introduction, for a system of \emph{linear equations} on $\Rn$ and the base ring is the ring of polynomials in $n$ unknowns,~\cite{FeffermanLuli2021Generators} provides an algorithm to compute the generators for the module of admissible forcings. Question~\ref{quest:effective algebraic generators}(A) is asking for the analogue for systems of inequalities. Again, the univariate and Nash assumptions here are more or less optimal, as we have seen from Theorem~\ref{thm:quotient module polynomial not FG} and Example~\ref{example:qjet}. On the other hand, Question~\ref{quest:effective algebraic generators}(B) takes us to a different direction, which calls for a more algebraic argument beyond the methods developed in this paper. Its \emph{equations} analogue has been extensively studied. See, e.g.,~\cite{kollarnowak2015rational,EpsteinHochster2018,kucharzkurdyka2017linear}.

\begin{remark}
\label{rem:saturation and algebraization} 

The preceding questions may also be equivalently expressed using the $\sa^m$-saturation. Abstractly, let $\module\subseteq\ring(I)^N $ be an $\ring(I)$-module, and let $\qmodule\subseteq\ring(I)^N$ be a semimodule over $\ring_+(I)$. 
 We define their extensions to $\sa^m(I,\R)$ by 
 \begin{equation*} 
\begin{split}
     \sa^m(I)\module &:= \set{ 
 \sum_{\ell=1}^k\varphi_\ell m_\ell: 
 \varphi_\ell\in\sa^m(I,\R),\ 
 m_\ell\in \module,\ k\in\N 
 } 
 \\
 \sa_+^m(I)\qmodule &:= \set{ \sum_{\ell=1}^k\psi_\ell q_\ell: \psi_\ell\in\sa_+^m(I),\ q_\ell\in\qmodule,\ k\in\N }.
\end{split}
 \end{equation*} 
The $\sa^m$-saturation of the pair $(\module,\qmodule)$ is defined by 
 \begin{equation*} 
 \sat_{\sa^m,I}(\module,\qmodule) 
 := \ring(I)^N \cap 
 \brac{\sa^m(I)\module + \sa_+^m(I)\qmodule}.
 \label{eq:Cm saturation} 
 \end{equation*} 

The unresolved issue in Question~\ref{quest:effective algebraic generators}(A) is the effective computation of the lineality module and of the relations among these coordinate classes. 

The polynomial Question~\ref{quest:effective algebraic generators}(B) contains an additional algebraic problem. Even when $A$ and $f$ are polynomial, the selection and slack variable produced by the arguments of this paper may bnot be a polynomial, so the saturation in \eqref{eq:Cm saturation} may be a nontrivial operation. 
\end{remark}

\section*{Acknowledgment}

The first author is supported by a grant from the City University of Hong Kong (Project No.\ 7200844). The second author is supported in part by NSF grant 2247429, the Chancellor's Fellowship at UC Davis, and the Simons Gift Fund.

\bibliographystyle{plain}
\bibliography{bib}

\end{document}